\documentclass[12pt]{amsart}
\usepackage{amssymb}
\usepackage{hyperref}

\newtheorem{theorem}{Theorem}[section]
\newtheorem{definition}[theorem]{Definition}
\newtheorem{proposition}[theorem]{Proposition}

\newtheorem{problem}[theorem]{Problem}
\newtheorem{question}[theorem]{Question}
\newtheorem{conjecture}[theorem]{Conjecture}

\begin{document}

\title{Analytic aspects of the circulant Hadamard conjecture}

\author{Teodor Banica}
\address{TB: Department of Mathematics, Cergy-Pontoise University, 95000 Cergy-Pontoise, France. {\tt teo.banica@gmail.com}}

\author{Ion Nechita}
\address{IN: CNRS, Laboratoire de Physique Th\'eorique, IRSAMC, Universit\'e de Toulouse, UPS, 31062 Toulouse, France. {\tt nechita@irsamc.ups-tlse.fr}}

\author{Jean-Marc Schlenker}
\address{JMS: University of Luxembourg, Campus Kirchberg, Mathematics Research Unit, 6 rue Richard Coudenhove-Kalergi, L-1359 Luxembourg. {\tt jean-marc.schlenker@uni.lu}}


\begin{abstract}
We investigate the problem of counting the real or complex Hadamard matrices which are circulant, by using analytic methods. Our main observation is the fact that for $|q_0|=\ldots=|q_{N-1}|=1$ the quantity $\Phi=\sum_{i+k=j+l}\frac{q_iq_k}{q_jq_l}$ satisfies $\Phi\geq N^2$, with equality if and only if $q=(q_i)$ is the eigenvalue vector of a rescaled circulant complex Hadamard matrix. This suggests three analytic problems, namely: (1) the brute-force minimization of $\Phi$, (2) the study of the critical points of $\Phi$, and (3) the computation of the moments of $\Phi$. We explore here these questions, with some results and conjectures.
\end{abstract}

\maketitle

\tableofcontents

\section*{Introduction}

An Hadamard matrix is a square matrix $H\in M_N(\pm 1)$, whose rows are pairwise orthogonal. The size of such a matrix must be $N=2$ or $N\in 4\mathbb N$, and the celebrated Hadamard Conjecture (HC) states that for any $N\in 4\mathbb N$, there exists an Hadamard matrix of order $N$. Also famous is the Circulant Hadamard Conjecture (CHC):

\medskip

\noindent {\bf Conjecture (CHC).} {\em There are exactly $8$ circulant Hadamard matrices, namely
$$K_4=\begin{pmatrix}-1&1&1&1\\ 1&-1&1&1\\ 1&1&-1&1\\ 1&1&1&-1\end{pmatrix}$$
and its $7$ conjugates, obtained from $\pm K_4$ by cyclic permutations of the rows.}

\medskip

We refer to the monographs \cite{aga}, \cite{hor}, \cite{rys} for a discussion of these conjectures. In what follows we will be mainly interested in the CHC, and its generalizations.

There has been a lot of work on the CHC, starting with Turyn's paper \cite{tur}. See \cite{jlo}, \cite{lsc}, \cite{sch}. In the case of the complex Hadamard matrices, the circulant problematics is interesting as well, due to Bj\"orck's reformulation in terms of cyclic $N$-roots \cite{bjo}. For a number of results and applications here, see \cite{alm}, \cite{bfr}, \cite{bha}, \cite{ckh}, \cite{goy}, \cite{ha2}, \cite{hjo}, \cite{sz1}, \cite{sz2}.

The aim of this paper is present a new point of view on the CHC, and on the related circulant complex Hadamard matrix problematics, by using analytic methods.

Let us fix $N\in\mathbb N$, and denote by $F=(w^{ij})/\sqrt{N}$ with $w=e^{2\pi {\rm i}/N}$ the Fourier matrix. It is well-known that $U\in M_N(\mathbb C)$ is circulant if and only if $U=FQF^*$, with $Q$ diagonal. In addition, for $U\in U_N$ the eigenvalues $q_i=Q_{ii}$ must be of modulus $1$.

On the other hand, the unitaries $U\in U_N$ which are multiples of complex Hadamard matrices can be detected by using the Cauchy-Schwarz inequality. For instance, we have $||U||_4\geq 1$, with equality if and only if $H=\sqrt{N}U$ is Hadamard. Now by restricting attention to the circulant matrices, and applying the Fourier transform, we obtain:

\medskip

\noindent {\bf Fact.} {\em For any complex numbers satisfying $|q_0|=\ldots=|q_{N-1}|=1$ we have
$$\sum_{i+k=j+l}\frac{q_iq_k}{q_jq_l}\geq N^2$$
with equality iff $H=\sqrt{N}FQF^*$, with $Q=diag(q_i)$, is a complex Hadamard matrix.}

\medskip

Here, and in what follows, all the indices are taken modulo $N$.

This observation suggests that the circulant complex Hadamard matrices can be studied as well by using analytic methods, and more precisely by further exploring the inequality appearing above. We will present in this paper some results in this direction.

Now, let us go back to the real case. So, let $U\in U_N$ be circulant, and write $U=FQF^*$ as above. It is easy to see that $U$ is a real matrix if and only if the eigenvalues $q_0,\ldots,q_{N-1}$ satisfy $\bar{q}_i=q_{-i}$ for any $i$, so the CHC has the following reformulation:

\medskip

\noindent {\bf Conjecture (CHC, analytic formulation).} {\em For any $N\neq 4$ the quantity
$$\Phi=\sum_{i+j+k+l=0}q_iq_jq_kq_l$$
satisfies $\Phi>N^2$, for any $q_0,\dots,q_{N-1}$ satisfying $|q_i|=1$ and $\bar{q}_i=q_{-i}$.}

\medskip

We will explore here various analytic methods that can be used for minimizing $\Phi$. These methods fall into three classes:
\begin{enumerate}
\item Direct approach. We will present here a few conjectural statements, found by solving the problem at $N=8,12$, using a computer. One open problem here is to decide whether $\min\Phi-N^2$ converges to $0$ or not, with $N\to\infty$.

\item Critical points. Here we can use a whole machinery from \cite{bc1}, \cite{bne}, \cite{bnz}. We will find here an algebraic criterion for the critical points, do some explicit computations for $N$ small, and state a conjecture found by using a computer.

\item The moment method. Our remark here is that the quantity $\widetilde{\Phi}$ obtained from $\Phi$ by dropping the arithmetic condition $i+j+k+l=0$ describes a certain random walk on the lattice $\mathbb Z^N\subset\mathbb R^N$, related to the P\'olya random walk \cite{pol}.
\end{enumerate}

Summarizing, regardless of the chosen approach, the concrete questions that we are led into look rather difficult, but definitely further approachable.

We should mention that this kind of approach, in the general context of the Hadamard matrix problematics, goes back to \cite{bc1}. Proposed there was an analytic approach to the HC, by using integrals over $O_N$. However, while a considerable quantity of work has gone into the study of such integrals, cf. \cite{bc2}, \cite{csn}, \cite{gor}, \cite{psw}, the state-of-art of the subject is several levels below what would be needed for any slight advance on the HC. Regarding the CHC, however, the situation here is quite different, the problem being probably easier. We intend to come back to these questions in some future work.

The paper is organized as follows: 1-2 contain basic facts on the circulant Hadamard matrices, in 3-4 we discuss in detail the norm estimate $\Phi\geq N^2$, and in 5-6 we explore the symmetries of the critical points of $\Phi$, and the computation of the moments of $\Phi$. 

\medskip

\noindent {\bf Acknowledgements.} We would like to thank Yemon Choi, Beno\^ it Collins, Rob Craigen and Karol \.Zyczkowski for several useful discussions.

\section{Circulant Hadamard matrices}

We consider in this paper various $N\times N$ matrices over the real or complex numbers. The matrix indices will range in $\{0,1,\ldots,N-1\}$, and will be taken modulo $N$.

As explained in the introduction, we are interested in the Circulant Hadamard Conjecture (CHC), stating that the only circulant Hadamard matrices can appear at $N=4$.

The first result in this direction, due to Turyn \cite{tur}, is as follows:

\begin{proposition}
The size of a circulant Hadamard matrix $H\in M_N(\pm 1)$ must be of the form $N=4n^2$, with $n\in\mathbb N$.
\end{proposition}

\begin{proof}
Let $a,b\in\mathbb N$ with $a+b=N$ be the number of $1,-1$ entries in the first row of $H$. If we denote by $H_0,\ldots,H_{N-1}$ the rows of $H$, then by summing over columns we get:
$$\sum_{i=0}^{N-1}<H_0,H_i>=a(a-b)+b(b-a)=(a-b)^2$$

On the other hand, the quantity on the left is $<H_0,H_0>=N$. Thus $N$ is a square, and together with the well-known fact that $N\in 2\mathbb N$, this gives $N=4n^2$, with $n\in\mathbb N$.
\end{proof}

Also found by Turyn in \cite{tur} is the fact that the above number $n\in\mathbb N$ must be odd, and not a prime power. For further generalizations of these results, see \cite{jlo}, \cite{lsc}, \cite{sch}.

An interesting extension of the HC/CHC problematics, that we would like to explain now, appears when looking at the complex Hadamard matrices:

\begin{definition}
A complex Hadamard matrix is a square matrix $H\in M_N(\mathbb C)$ whose entries are on the unit circle, $|H_{ij}|=1$, and whose rows are pairwise orthogonal.
\end{definition}

The basic example here is the rescaled Fourier matrix, $F_N=(w^{ij})$ with $w=e^{2\pi {\rm i}/N}$:
$$F_N=\begin{pmatrix}
1&1&1&\ldots&1\\
1&w&w^2&\ldots&w^{N-1}\\
\ldots&\ldots&\ldots&\ldots&\ldots\\
1&w^{N-1}&w^{2(N-1)}&\ldots&w^{(N-1)^2}
\end{pmatrix}$$

This example prevents the existence of a complex analogue of the HC. However, when trying to construct complex Hadamard matrices with roots of unity of a given order $l\in\mathbb N$, a subtle generalization of the HC problematics appears. See \cite{ha1}, \cite{lle}, \cite{lau}, \cite{tzy}, \cite{win}.

For a number of applications of such matrices, see \cite{bb+}, \cite{jsu}, \cite{pop}, \cite{ta1}, \cite{wer}.

Regarding now the circulant case, once again there is much more ``room'' in the complex case. Here is a basic example of a circulant complex Hadamard matrix:
$$\widetilde{F}_2=\begin{pmatrix}1&{\rm i}\\ {\rm i}&1\end{pmatrix}$$

Here the notation $\widetilde{F}_2$ comes from the fact that this matrix can be obtained from the Fourier matrix $F_2$ by performing certain elementary operations:

\begin{definition}
Two complex Hadamard matrices $H,K\in M_N(\mathbb C)$ are called equivalent if one can pass from one to the other by permuting the rows and columns, or by multiplying the rows and columns by complex numbers of modulus $1$.
\end{definition}

In other words, we use the equivalence relation on the $N\times N$ matrices coming from the action $T_{(A,B)}(H)=AHB^*$ of the group $G=((\mathbb T\wr S_N)\times(\mathbb T\wr S_N))/\mathbb T$. Here, and in what follows, we denote by $\mathbb T$ the unit circle in the complex plane.

Now back to the circulant examples, at $N=3$ a similar situation happens, because the Fourier matrix $F_3$ is equivalent to the following circulant matrix ($w=e^{2\pi {\rm i}/3}$):
$$\widetilde{F}_3=\begin{pmatrix}1&1&w\\ w&1&1\\ 1&w&1\end{pmatrix}$$

At $N=4$ we have the matrix $K_4$ appearing in the statement of the CHC. Note that $K_4$ is equivalent to $F_2\otimes F_2$, which is the Fourier matrix of the Klein group $\mathbb Z_2\otimes \mathbb Z_2$. 

At $N=5$ now, or more generally at any $N$ odd, we have the following matrix:
$$\tilde{F}_N=\begin{pmatrix}1&1&w&w^3&w^6&\ldots&w^{\frac{(N-2)(N-1)}{2}}\\
w^{\frac{(N-2)(N-1)}{2}}&1&1&w&w^3&\ldots&w^{\frac{(N-3)(N-2)}{2}}\\
\ldots&\ldots&\ldots&\ldots&\ldots&\ldots&\ldots\\
1&w&w^3&w^6&w^{10}&\ldots&1\end{pmatrix}$$

We will be back to these basic examples a bit later.

There are many other examples of circulant Hadamard matrices. Here is an ``exotic'' one found by Bj\"orck and Fr\"oberg in \cite{bfr}, using a root of $a^2-(1-\sqrt{3})a+1=0$:
$$BF_6=\begin{pmatrix}
1&{\rm i}a&-a&-{\rm i}&-\bar{a}&{\rm i}\bar{a}\\
{\rm i}\bar{a}&1&{\rm i}a&-a&-{\rm i}&-\bar{a}\\
-\bar{a}&{\rm i}\bar{a}&1&{\rm i}a&-a&-{\rm i}\\
-{\rm i}&-\bar{a}&{\rm i}\bar{a}&1&{\rm i}a&-a\\
-a&-{\rm i}&-\bar{a}&{\rm i}\bar{a}&1&{\rm i}a\\
{\rm i}a&-a&-{\rm i}&-\bar{a}&{\rm i}\bar{a}&1
\end{pmatrix}$$

There has been a lot of interesting work on the circulant complex Hadamard matrices, and their various applications, see \cite{alm}, \cite{bjo}, \cite{bfr}, \cite{bha}, \cite{ckh}, \cite{ha2}, \cite{hjo}, \cite{sz1}, \cite{sz2}.

Let us first record the following reformulation of the problem, due to Bj\"orck \cite{bjo}:

\begin{proposition}
Assume that $H\in M_N(\mathbb T)$ is circulant, $H_{ij}=\xi_{j-i}$. Then $H$ is Hadamard if and only if the vector $(z_0,z_1,\ldots,z_{N-1})$ given by $z_i=\xi_i/\xi_{i-1}$ satisfies:
\begin{eqnarray*}
z_0+z_1+\ldots+z_{N-1}&=&0\\
z_0z_1+z_1z_2+\ldots+z_{N-1}z_0&=&0\\
\ldots\\
z_0z_1\ldots z_{N-2}+\ldots+z_{N-1}z_0\ldots z_{N-3}&=&0\\
z_0z_1\ldots z_{N-1}&=&1
\end{eqnarray*}
If so is the case, we say that $z=(z_0,\ldots,z_{N-1})$ is a cyclic $N$-root.
\end{proposition}

\begin{proof}
Observe first that the last equation, namely $z_0z_1\ldots z_{N-1}=1$, is trivially satisfied. With the notation $\xi_0=\lambda$, the vector $\xi$ is given by:
$$\xi=(\lambda,\lambda z_0,\lambda z_0z_1,\lambda z_0z_1z_2,\ldots,\lambda z_0z_1\ldots z_{N-2})$$

Now by writing down the orthogonality conditions between the rows of $H$, we see that these correspond precisely to the above collection of $=0$ equations. See \cite{bjo}.
\end{proof}

As a basic application, let $w=e^{2\pi{\rm i}/N}$ and set $z_i=w^i$. Then the above $=0$ equations for a cyclic $N$-root are all satisfied, and we have:
$$z_0z_1\ldots z_{N-1}=w^{\frac{N(N-1)}{2}}=e^{\pi {\rm i}(N-1)}=(-1)^{N-1}$$ 

Thus when $N$ is odd we obtain a cyclic $N$-root, and hence a circulant complex Hadamard matrix, which is the matrix $\widetilde{F}_N$ given above. A similar method works for $N$ even, since multiplying each $z_i=w^i$ by $e^{\pi{\rm i}/N}$ makes the $=1$ condition to be satisfied.

Let us record as well the following result, due to Haagerup \cite{ha2}:

\begin{theorem}
When $N$ is prime, the number of circulant $N\times N$ complex Hadamard matrices, counted with certain multiplicities, is exactly $\binom{2N-2}{N-1}$.
\end{theorem}

\begin{proof}
The proof in \cite{ha2} uses some further manipulations of the cyclic $N$-root equations, and then a number of algebraic geometry and number theory ingredients, notably including, in order to prove the finiteness claim, a theorem of Chebotarev, which states that when $N$ is prime, all the minors of the Fourier matrix $F_N$ are nonzero. See \cite{ha2}.
\end{proof}

Finally, let us discuss the relationship between the real and the complex circulant Hadamard matrix problematics. In order to do so, we use the following key notion:

\begin{definition}
The Butson class $C_N(l)$, with $l\in\{2,3,\ldots,\infty\}$, consists of the $N\times N$ complex Hadamard matrices having as entries the $l$-th roots of unity.
\end{definition}

These matrices were introduced by Butson in \cite{but}. Observe that $C_N(2)$ is the set of all $N\times N$ Hadamard matrices, and that $C_N(\infty)$ is the set of all $N\times N$ complex Hadamard matrices. Observe also that if we denote by $\mathbb Z_l$ the group of $l$-th roots of unity, then:
$$C_N(l)=M_N(\mathbb Z_l)\cap\sqrt{N}U_N$$

As already mentioned, introducing these matrices leads to the correct complex generalization of the HC problematics. More precisely, the big problem is to characterize, at least conjecturally, the pairs of integers $(N,l)$ such that $C_N(l)\neq\emptyset$. See \cite{lle}, \cite{lau}, \cite{win}.

Regarding now the CHC problematics, the question here is of course:

\begin{problem}
What is the number of circulant matrices in $C_N(l)$?
\end{problem}

Observe that at $l=2,\infty$ this problem corresponds respectively to the CHC, and to the generalization of Theorem 1.5 above. Note that at $l=\infty$, for values of $N\in\mathbb N$ which are multiples of squares, the answer is $\infty$, due to a result of Backelin. See \cite{bak}, \cite{fau}.

At other values $l=3,4,\ldots$, very little seems to be known about this problem. For instance it is quite unclear to us for which finite abelian groups $G=\mathbb Z_{N_1}\times\ldots\times\mathbb Z_{N_k}$ the corresponding Fourier matrices $F_G=F_{N_1}\otimes\ldots\otimes F_{N_k}$ can be put in circulant form. This kind of result, which looks quite elementary, would provide us at least with some examples of pairs $(N,l)$ for which $C_N(l)$ contains at least one circulant matrix. 

\section{Butson matrices}

In this section we further discuss Problem 1.7, by paying special attention to the simpler question on whether the set of circulant elements of $C_N(l)$ is empty or not. That is, we would like to understand what the $l$-analogue of the CHC should be.

We will need a number of facts on the vanishing sums of roots of unity. We call ``cycle'' any sum of the form $\lambda+\lambda w+\lambda w^2+\ldots+\lambda w^{p-1}$, with $|\lambda|=1$ and $w=e^{2\pi {\rm i}/p}$, with $p\in\mathbb N$ prime. The following theorem is basically due to Lam and Leung \cite{lle}:

\begin{theorem}
Let $l=p_1^{a_1}\ldots p_k^{a_k}$, and assume that $\lambda_i\in\mathbb Z_l$ satisfy $\lambda_1+\ldots+\lambda_N=0$.
\begin{enumerate}
\item $\sum\lambda_i$ is a sum of cycles, with $\mathbb Z$ coefficients.

\item If $k\leq 2$ then $\sum\lambda_i$ is a sum of cycles (with $\mathbb N$ coefficients).

\item At $k\geq 3$ then $\sum\lambda_i$ might not decompose as a sum of cycles.

\item $\sum\lambda_i$ has the same length as a sum of cycles: $N\in p_1\mathbb N+\ldots+p_k\mathbb N$.
\end{enumerate}
\end{theorem}

\begin{proof}
Here are a few explanations of these results:

(1) This is clear at $k=1$, and is quite elementary as well at $k=2$. See \cite{lle}.

(2) This is a well-known result, which follows from basic number theory.

(3) The simplest counterexample is as follows, with $w=e^{2\pi {\rm i}/30}$:
$$w^5+w^6+w^{12}+w^{18}+w^{24}+w^{25}=0$$

(4) This is a deep result, due to Lam and Leung \cite{lle}.
\end{proof}

In terms of the Butson matrices now, we have:

\begin{proposition}
If $C_N(l)\neq\emptyset$ then the following hold:
\begin{enumerate}
\item Lam-Leung obstruction: $l=p_1^{a_1}\ldots p_k^{a_k}$ implies $N\in p_1\mathbb N+\ldots+p_k\mathbb N$.

\item de Launey obstruction: there exists $d\in\mathbb Z[e^{2\pi{\rm i}/l}]$ such that $|d|^2=N^N$.
\end{enumerate}
\end{proposition}

\begin{proof}
Here (1) comes Theorem 2.1. As for (2), this comes from taking $d=\det H$, with the corresponding obstruction being studied by de Launey in \cite{lau}. 
\end{proof}

One interesting problem is whether the above two obstructions can be improved or not, in the particular case of the circulant matrices. This doesn't seem to be the case for the de Launey obstruction. As for the Lam-Leung obstruction, the situation here is quite unclear. Indeed, take $H\in C_N(l)$, and let us look at its first two rows:
$$H=\begin{pmatrix}
w^{r_0}&w^{r_1}&\ldots&w^{r_{N-1}}\\
w^{r_{N-1}}&w^{r_0}&\ldots&w^{r_{N-2}}\\
\ldots&\ldots&\ldots&\ldots
\end{pmatrix}$$

The vanishing scalar product between these two rows is then:
$$w^{r_0-r_{N-1}}+w^{r_1-r_0}+\ldots+w^{r_{N-1}-r_{N-2}}=0$$

Thus we have indeed an equation of type $\lambda_1+\ldots+\lambda_N=0$ with $\lambda_i\in\mathbb Z_l$, but the point is that these numbers $\lambda_i$ satisfy the extra condition $\lambda_1\ldots\lambda_N=1$. So, we are led into:

\begin{problem}
What are the possible lengths $N$ of the sums of type $\lambda_1+\ldots+\lambda_N=0$, with $\lambda_i\in\mathbb Z_l$ satisfying $\lambda_1\ldots\lambda_N=1$?
\end{problem}

As a first remark, when $l$ is odd the extra condition $\lambda_1\ldots\lambda_N=1$ won't change the Lam-Leung condition $N\in p_1\mathbb N+\ldots+p_k\mathbb N$, because the product of the elements of $\mathbb Z_p$ with $p$ odd is 1, so we can construct as sum as above, for any $N\in p_1\mathbb N+\ldots+p_k\mathbb N$.

The same observation holds in the case $4|l$, because we have ${\rm i}(-{\rm i})=1$, so we can construct as sum as above simply by choosing all needed $2$-cycles to be ${\rm i}+(-{\rm i})=0$.

So, the problem really makes sense in the case $l=2L$, with $L$ odd. In the simplest case $l=2$, we already have a change, because the Lam-Leung obstruction, namely $N\in 2\mathbb N$, becomes $N\in 4\mathbb N$ when adding the extra assumption $\lambda_1\ldots\lambda_N=1$. The next interesting case is $l=2p^a$, where the Lam-Leung obstruction is:
$$N\in 2\mathbb N+p\mathbb N=\{2,4,6,\ldots,p-1\}\cup\{p,p+1,p+2,\ldots\}$$

Observe that, by using Theorem 2.1 (2), we just have to compute the possible lengths $N$ of the sums of cycles, under the assumption $\lambda_1\ldots\lambda_N=1$. With this observation in hand, it is clear that the set $\{2,4,6,\ldots,p-1\}$ appearing above will be replaced by the set $\{4,8,12,\ldots,4[\frac{p-1}{4}]\}$. As for the set $\{p,p+1,p+2,\ldots\}$, the situation here is quite unclear. For instance $p$ is definitely allowed, $p+1$ is allowed precisely when it is a multiple of $4$, while $p+2$ is allowed as well, but due to the following computation:
$$(w+w^{2p^{a-1}+1}+w^{4p^{a-1}+1}+\ldots+w^{2(p-1)p^{a-1}+1})+(w^{\frac{p^a-p}{2}}+w^{p^a+\frac{p^a-p}{2}})=0+0=0$$
$$(w\cdot w^{2p^{a-1}+1}\cdot w^{4p^{a-1}+1}\ldots w^{2(p-1)p^{a-1}+1})(w^{\frac{p^a-p}{2}}\cdot w^{p^a+\frac{p^a-p}{2}})=w^p\cdot w^{-p}=1$$

So, the above problem is quite non-trivial, even in the simplest case $l=2p^a$!

As a conclusion, the general obstructions in Proposition 2.2 above apply of course to the circulant Butson matrix case, with the remark that the Lam-Leung obstruction can be probably slightly improved in this case, for exponents of type $l=2L$ with $L$ odd.

Let us discuss now a third obstruction, which is this time circulant matrix-specific. We recall from Proposition 1.1 above, due to Turyn \cite{tur} that at $l=2$ the matrix size $N$ must be a square. Here is the straightforward generalization of this fact:

\begin{proposition}
Assume that $H\in C_N(l)$ is circulant, let $w=e^{2\pi {\rm i}/l}$. If $a_0,\ldots,a_{l-1}\in\mathbb N$ with $\sum a_i=N$ are the number of $1,w,\ldots,w^{l-1}$ entries in the first row of $H$, then:
$$\sum_{ik}w^ka_ia_{i+k}=N$$
This condition, with $\sum a_i=N$, will be called ``Turyn obstruction'' on $(N,l)$.
\end{proposition}

\begin{proof}
Indeed, by summing over the columns of $H$, we obtain:
$$\sum_i<H_0,H_i>=\sum_{ij}<w^i,w^j>a_ia_j=\sum_{ij}w^{i-j}a_ia_j$$

Now since the left term is $<H_0,H_0>=N$, this gives the result.
\end{proof}

\begin{proposition}
When $l$ is prime, the Turyn obstruction is $\sum_i(a_i-a_{i+k})^2=2N$ for any $k\neq 0$. Also, for small values of $l$, the Turyn obstruction is as follows:
\begin{enumerate}
\item At $l=2$ the condition is $(a_0-a_1)^2=N$.

\item At $l=3$ the condition is $(a_0-a_1)^2+(a_1-a_2)^2+(a_2-a_3)^2=2N$.

\item At $l=4$ the condition is $(a_0-a_2)^2+(a_1-a_3)^2=N$.

\item At $l=5$ the condition is $\sum_i(a_i-a_{i+1})^2=\sum_i(a_i-a_{i+2})^2=2N$.
\end{enumerate}
\end{proposition}

\begin{proof}
We use the well-known fact that when $l$ is prime, the vanishing sums of $l$-roots of unity are exactly the sums of type $c+cw+\ldots+cw^{l-1}$, with $c\in\mathbb N$. Thus the Turyn obstruction is equivalent to the following equations, one for each $k\neq 0$:
$$\sum_ia_i^2-\sum_ia_ia_{i+k}=N$$

Now by forming squares, this gives the equations in the statement.

Regarding now the $l=2,3,4,5$ assertions, these follow from the first assertion when $l$ is prime, $l=2,3,5$. Also, at $l=4$ we have $w={\rm i}$, so the Turyn obstruction reads:
$$(a_0^2+a_1^2+a_2^2+a_3^2)+{\rm i}\sum a_ia_{i+1}-2(a_0a_2+a_1a_3)-{\rm i}\sum a_ia_{i+1}=N$$

Thus the imaginary terms cancel, and we obtain the formula in the statement.
\end{proof}

It is quite unclear what the final statement of the Turyn obstruction is, for composite exponents $l>4$. The simplest combinatorics appears in the case $l=p^2$ with $p>3$ prime, so let us work out explicitely what happens at $l=9$. According to Theorem 2.1 (2) any vanishing sum of 9-roots of unity must split as a sum of cycles, as follows:
$$c_1(1+w^3+w^6)+c_2(w+w^4+w^7)+c_3(w^2+w^5+w^8)=0$$

Here $w=e^{2\pi {\rm i}/9}$. Now, with this observation in hand, let us go back to the Turyn obstruction. With $A_k=\sum_ia_ia_{i+k}$, the equation is $\sum_kw^kA_k=N$, which gives:
$$\begin{cases}
A_0-N&=A_3=A_6\\
\ \ \ \ A_1&=A_4=A_7\\
\ \ \ \  A_2&=A_5=A_8
\end{cases}$$

Let us look at the first equation. Since we have $A_3=A_6$ we can delete the last equality, and the equation becomes $A_0-N=A_3$, which looks as follows:
$$(\sum a_i^2)-N=a_0a_3+a_1a_4+a_2a_5+a_3a_6+a_4a_7+a_5a_8+a_6a_0+a_7a_1+a_8a_2$$

Now by doubling and forming squares, this equation is equivalent to:
\begin{eqnarray*}
&&(a_0-a_3)^2+(a_3-a_6)^2+(a_6-a_0)^2\\
&&+(a_1-a_4)^2+(a_4-a_7)^2+(a_7-a_1)^2\\
&&+(a_2-a_5)^2+(a_5-a_8)^2+(a_8-a_2)^2=2N
\end{eqnarray*}

As an example, let us take $N=6$. The possible solutions can only come from:
$$12=9+1+1+1+0+0+0+0+0$$ 
$$12=4+1+1+1+1+1+1+1+1$$
$$12=4+4+1+1+1+1+0+0+0$$

The first two cases are excluded, because we cannot group the terms in 3 groups of 3 terms each, with the sum 0 in each group. As for the third case, this leads to:
$$12=(2^2+(-1)^2+(-1)^2)+(2^2+(-1)^2+(-1)^2)+(0^2+0^2+0^2)$$

Thus, our equation $A_0-N=A_3$, when combined with the condition $\sum a_i=N$, tells us that among the sets with repetition $\{a_0,a_3,a_6\}$, $\{a_1,a_4,a_7\}$, $\{a_2,a_5,a_8\}$, two of these sets must be equal to $\{0,1,2\}$, and the remaining set, to $\{0,0,0\}$. The problem now is to decide if the remaining equations $A_1=A_4=A_7$ and $A_2=A_5=A_8$ can be satisfied or not, and this doesn't look trivial. Summarizing, even the simplest possible non-trivial application of the Turyn obstruction for $l>4$ composite seems to require a computer.

As a conclusion here, let us raise the following question:

\begin{question}
What is the simplest statement of the Turyn obstruction, for exponents $l>4$ which are not prime, and have at most $2$ prime factors?
\end{question} 

Observe that we have included here the $k\leq 2$ assumption appearing in Theorem 2.1 (2) above. Of course, at $k\geq 3$ the problem looks extremely complicated.

Let us discuss now the existence of circulant Butson matrices, for small $N,l$:

\begin{theorem}
We have the following table, where $\circ,\circ_t$ are the Lam-Leung and Turyn obstructions, and where the crosses are exactly where $C_N(l)^{circ}\neq\emptyset$:
\end{theorem}
\begin{center}
\begin{tabular}[t]{|l|l|l|l|l|l|l|l|l|l|l|l|l|l|l|l||||}
\hline $N\backslash l$&2&3&4&5&6&7&8&9\\
\hline 2&$\circ_t$&$\circ$&x&$\circ$&$\circ_t$&$\circ$&x&$\circ$\\
\hline 3&$\circ$&x&$\circ$&$\circ$&x&$\circ$&$\circ$&x\\
\hline 4&x&$\circ$&x&$\circ$&x&$\circ$&x&$\circ$\\
\hline 5&$\circ$&$\circ$&$\circ$&x&$\circ_t$&$\circ$&$\circ$&$\circ$\\
\hline 6&$\circ_t$&$\circ_t$&$\circ_t$&$\circ$&$\circ_t$&$\circ$&&$\circ_t$\\
\hline 7&$\circ$&$\circ$&$\circ$&$\circ$&&x&$\circ$&$\circ$\\
\hline 8&$\circ_t$&$\circ$&x&$\circ$&$\circ_t$&$\circ$&x&$\circ$\\
\hline 9&$\circ$&x&$\circ$&$\circ$&x&$\circ$&$\circ$&x\\
\hline
\end{tabular}
\end{center}

\begin{proof}
First, the fact that the crosses are exactly where they should be was found by implementing the circulant Hadamard matrix condition on a computer.

Regarding now the various obstructions:

(a) The Lam-Leung obstruction was applied first, and led to the $\circ$ symbols in the above table. As an observation, the possibly improved obstruction, coming from the condition $\lambda_1\ldots\lambda_N=1$ discussed above, doesn't apply in the range $(N,l)\in\{2,\ldots,9\}^2$.

(b) The Turyn obstruction was applied next, to the empty squares left. Here is the detail, using notations from Proposition 2.5 above:

-- (2,2),(6,2),(8,2). At $l=2$ the Turyn obstruction is simply $N=(a_0-a_1)^2$, excluding indeed the values $N=2,6,8$.

-- (2,6). At $N=2$ we must have $a_0+\ldots+a_{l-1}=2$, and hence there are two cases. The first case is $a_k=2$ for some $k$, but here the Turyn obstruction becomes $2\cdot 2=2$, contradiction. The second case is $a_k=a_s=1$ for some $s\neq t$, and here the obstruction becomes $1+1+w^{s-t}+w^{t-s}=2$, so $w^{s-t}=\pm {\rm i}$, and so $4|l$, excluding indeed $l=6$.

-- (6,3). At $l=3$, according to Proposition 2.5, the obstruction is $(a_0-a_1)^2+(a_1-a_2)^2+(a_2-a_0)^2=2N$. Now since the only way of writing $12$ as a sum of three squares is $12=4+4+4$, we reach to a contradiction, because $\pm2\pm2\pm2\neq 0$, for any choices of the $\pm$ signs. Thus the case $N=6$ is indeed excluded.

-- (6,4). At $l=4$, according to Proposition 2.5, the obstruction is $(a_0-a_2)^2+(a_1-a_3)^2=N$. But since $N=6$ is not a sum of squares, this value is indeed excluded.

-- (5,6),(6,6),(6,9),(8,6). Here the application of the Turyn obstruction is a more complicated task, and we obtained the results by using a computer.
\end{proof}

Regarding the two blank cases in the above table, we do not know how to deal with them: there are no matrices there, nor known obstructions which apply.

\section{Fourier formulation}

We fix $N\in\mathbb N$ and we denote by $F=(w^{ij})/\sqrt{N}$ with $w=e^{2\pi {\rm i}/N}$ the Fourier matrix. Observe that $F_N=\sqrt{N}F$ is the complex Hadamard matrix that we met in section 1.

Given a vector $q\in\mathbb C^N$, we denote by $Q\in M_N(\mathbb C)$ the diagonal matrix having $q$ as vector of diagonal entries. That is, $Q_{ii}=q_i$, and $Q_{ij}=0$ for $i\neq j$.

We will make a heavy use of the following well-known result:

\begin{proposition}
The various sets of circulant matrices are as follows:
\begin{enumerate}
\item $M_N(\mathbb C)^{circ}=\{FQF^*|q\in\mathbb C^N\}$.

\item $U_N^{circ}=\{FQF^*|q\in\mathbb T^N\}$.

\item $O_N^{circ}=\{FQF^*|q\in\mathbb T^N,\bar{q}_i=q_{-i},\forall i\}$.
\end{enumerate}
In addition, the first row vector of $FQF^*$ is given by $\xi=Fq/\sqrt{N}$.
\end{proposition}

\begin{proof}
This is well-known, but since we will often use it, here is the proof:

(1) If $H_{ij}=\xi_{j-i}$ is circulant then $Q=F^*HF$ is diagonal, given by:
$$Q_{ij}=\frac{1}{N}\sum_{kl}w^{jl-ik}\xi_{l-k}=\delta_{ij}\sum_rw^{jr}\xi_r$$ 

Also, if $Q=diag(q)$ is diagonal then $H=FQF^*$ is circulant, given by:
$$H_{ij}=\sum_kF_{ik}Q_{kk}\bar{F}_{jk}=\frac{1}{N}\sum_kw^{(i-j)k}q_k$$

Observe that this latter formula proves as well the last assertion,  $\xi=Fq/\sqrt{N}$.

(2) This is clear from (1), because the eigenvalues must be on the unit circle $\mathbb T$.

(3) Observe first that for $q\in\mathbb C^N$ we have $\overline{Fq}=F\tilde{q}$, with $\tilde{q}_i=\bar{q}_{-i}$, and so $\xi=Fq$ is real if and only if $\bar{q}_i=q_{-i}$ for any $i$. Together with (2), this gives the result.
\end{proof}

Observe that in (3), the equations for the parameter space are $q_0=\bar{q}_0$, $\bar{q}_1=q_{n-1}$, $\bar{q}_2=q_{n-2}$, and so on until $[N/2]+1$. Thus, with the convention $\mathbb Z_\infty=\mathbb T$ we have:
$$O_N^{circ}\simeq
\begin{cases}
\mathbb Z_2\times\mathbb Z_\infty^{{(N-1)}/2}&(N\ {\rm odd})\\
\mathbb Z_2^2\times\mathbb Z_\infty^{(N-2)/2}&(N\ {\rm even})
\end{cases}$$

In terms of circulant Hadamard matrices, we have the following statement:

\begin{proposition}
The sets of complex and real circulant Hadamard matrices are:
$$C_N(\infty)^{circ}=\{\sqrt{N}FQF^*|q\in\mathbb T^N\}\cap M_N(\mathbb T)$$
$$C_N(2)^{circ}=\{\sqrt{N}FQF^*|q\in\mathbb T^N,\bar{q}_i=q_{-i}\}\cap M_N(\pm1)$$
In addition, the sets of $q$ parameters are invariant under cyclic permutations, and also under mutiplying by numbers in $\mathbb T$, respectively under multiplying by $-1$. 
\end{proposition}

\begin{proof}
All the assertions are indeed clear from Proposition 3.1.
\end{proof}

In the above statement we have used of course the Butson matrix notations $C_N(2)$ and $C_N(\infty)$ for the sets of real and complex Hadamard matrices. In the general Butson matrix case the situation is quite unclear, and we have here the following question:

\begin{problem}
Consider the Butson class $C_N(l)=M_N(\mathbb Z_l)\cap\sqrt{N}U_N$.
\begin{enumerate}
\item Is there a group $U_N(l)\subset U_N$ such that $C_N(l)=M_N(\mathbb T)\cap\sqrt{N}U_N(l)$?

\item Is there a group $\mathbb T^N(l)\subset\mathbb T^N$ such that $C_N(l)^{circ}=M_N(\mathbb T)\cap\sqrt{N}F\mathbb T^N(l)F^*$?
\end{enumerate}
\end{problem}

The answer to these questions is of course clear at $l=2,\infty$. However, at $l=3,4$ already, and especially at $l=4$, the answer to these questions is quite unclear. 

Observe that a positive answer to the first question would imply a positive answer to the second question, because we could simply set $\mathbb T^N(l)=F^*U_N(l)^{circ}F$. 

However, our belief is that, in the general case, the answer to the first question should be rather ``no'', and to the second one, maybe ``yes''. We have no further results here.

Let us go back now to the complex case, where the parameter space is $\mathbb T^N$. We construct now a map $\Phi:\mathbb T^N\to (0,\infty)$, which will play a key role in what follows:

\begin{definition}
Associated to $q\in\mathbb T^N$, written $q=(q_0,\ldots,q_{N-1})$ is the quantity
$$\Phi=\sum_{i+k=j+l}\frac{q_iq_k}{q_jq_l}$$
where all the indices are taken modulo $N$.
\end{definition}

As a first observation, by conjugating the above expression we see that $\Phi$ is real. In fact, $\Phi$ is by definition a sum of $N^3$ terms, consisting of $N(2N-1)$ values of $1$ and of $N(N-1)^2$ other complex numbers of modulus 1, coming in pairs $(a,\bar{a})$. 

We will be back to these observations a bit later. For the moment, let us record the following key statement, which is the starting point for the investigations to follow:

\begin{theorem}
For $q\in\mathbb T^N$ we have $\Phi\geq N^2$, with equality if and only if $\sqrt{N}q$ is the eigenvalue vector of a circulant Hadamard matrix $H\in M_N(\mathbb C)$. 
\end{theorem}

\begin{proof}
For $U\in U_N$ we have the following Cauchy-Schwarz estimate:
$$||U||_4^4=\sum_{ij}|U_{ij}|^4\geq\frac{1}{N^2}\left(\sum_{ij}|U_{ij}|^2\right)^2=1$$

Thus we have $||U||_4\geq 1$, with equality if and only if $H=\sqrt{N}U$ is Hadamard. In the particular case of circulant matrices, by using the formula $\xi=Fq/\sqrt{N}$, we have:
\begin{eqnarray*}
||U||_4^4
&=&N\sum_s|\xi_s|^4\\
&=&\frac{1}{N^3}\sum_s|\sum_iw^{si}q_i|^4\\
&=&\frac{1}{N^3}\sum_s\sum_iw^{si}q_i\sum_jw^{-sj}\bar{q}_j\sum_kw^{sk}q_k\sum_lw^{-sl}\bar{q}_l\\
&=&\frac{1}{N^3}\sum_s\sum_{ijkl}w^{(i-j+k-l)s}\frac{q_iq_k}{q_jq_l}\\
&=&\frac{1}{N^2}\sum_{i+k=j+l}\frac{q_iq_k}{q_jq_l}
\end{eqnarray*}

Thus we have $\Phi=N^2||U||_4^4\geq N^2$, and we are done.
\end{proof}

The above result makes the link with our previous work in \cite{bc1}, \cite{bne}, \cite{bnz} on the $p$-norms over the orthogonal group, and perhaps also with some other analytic considerations, as those in \cite{grz}, \cite{llv}, \cite{ta2}. Of course, the passage through the $4$-norm is optional, and we have the following more direct explanation of the above result:

\begin{proposition}
We have the formula
$$\Phi=N^2+\sum_{i\neq j}(|\nu_i|^2-|\nu_j|^2)^2$$
where $\nu=(\nu_0,\ldots,\nu_{N-1})$ is the vector given by $\nu=Fq$.
\end{proposition}

\begin{proof}
This follows by replacing in the above proof the Cauchy-Schwarz estimate by the corresponding sum of squares. More precisely, we know from the above proof that:
$$\Phi=N^3\sum_i|\xi_i|^4$$

On the other hand $U_{ij}=\xi_{j-i}$ being unitary, we have $\sum_i|\xi_i|^2=1$, and so:
\begin{eqnarray*}
1
&=&\sum_i|\xi_i|^4+\sum_{i\neq j}|\xi_i|^2\cdot|\xi_j|^2\\
&=&N\sum_i|\xi_i|^4-\left((N-1)\sum_i|\xi_i|^4-\sum_{i\neq j}|\xi_i|^2\cdot|\xi_j|^2\right)\\
&=&\frac{1}{N^2}\Phi-\sum_{i\neq j}(|\xi_i|^2-|\xi_j|^2)^2
\end{eqnarray*}

Now by multiplying by $N^2$, this gives the formula in the statement.
\end{proof}

We will explore the minimization problem for $\Phi$ in the next sections, by using various methods. As an illustration for the difficulties in dealing with this problem, let us work out the case where $N$ is small. At $N=1$ our inequality $\Phi\geq N^2$ is simply:
$$\Phi=1\geq 1$$

At $N=2$ our inequality is also clearly true:
$$\Phi=6+\left(\frac{q_0}{q_1}\right)^2+\left(\frac{q_1}{q_0}\right)^2\geq 4$$

At $N=3$ now, the inequality is something more subtle:
$$\Phi=15+4Re\left(\frac{q_0^3+q_1^3+q_2^3}{q_0q_1q_2}\right)\geq 9$$

Observe that in terms of $a=q_0^2/(q_1q_2)$, $b=q_1^2/(q_0q_2)$, $c=q_2^2/(q_0q_1)$, which satisfy $|a|=|b|=|c|=1$ and $abc=1$, our function is $\Phi=15+4Re(a+b+c)$. Thus at $N=3$ our inequality still has a quite tractable form, namely:
$$|a|=|b|=|c|=1,abc=1\implies Re(a+b+c)\geq-\frac{3}{2}$$

At $N=4$ however, the formula of $\Phi$ is as follows:
\begin{eqnarray*}\Phi
&=&28+4\left(\frac{q_0q_1}{q_2q_3}+\frac{q_2q_3}{q_0q_1}+\frac{q_0q_3}{q_1q_2}+\frac{q_1q_2}{q_0q_3}\right)+\left(\frac{q_0^2}{q_2^2}+\frac{q_2^2}{q_0^2}+\frac{q_1^2}{q_3^2}+\frac{q_3^2}{q_1^2}\right)\\
&&+2\left(\frac{q_0q_2}{q_1^2}+\frac{q_1^2}{q_0q_2}+\frac{q_0q_2}{q_3^2}+\frac{q_3^2}{q_0q_2}+\frac{q_1q_3}{q_0^2}+\frac{q_0^2}{q_1q_3}+\frac{q_1q_3}{q_2^2}+\frac{q_2^2}{q_1q_3}\right)
\end{eqnarray*}

It is not clear how to obtain a simple direct proof of $\Phi\geq 16$.

\section{The minimization problem}

Let us discuss now the real case. Here the parameter space is $\{q\in\mathbb T^N|\bar{q}_i=q_{-i}\}$, and our first task is to exploit the extra symmetries of the problem:

\begin{proposition}
For $q\in\mathbb T^N$ satisfying $\bar{q}_i=q_{-i}$ we have:
$$\Phi=\sum_{i+j+k+l=0}q_iq_jq_kq_l$$
In addition, we have $\Phi(q)=\Phi(-q)=\Phi(\tilde{q})=\Phi(-\tilde{q})$, where $\tilde{q}_i=w^iq_i$.
\end{proposition}

\begin{proof}
The first assertion is clear from definitions, because we have:
$$\Phi=\sum_{i+k=j+l}\frac{q_iq_k}{q_jq_l}=\sum_{i+k=j+l}q_iq_k\bar{q}_j\bar{q}_l=\sum_{i+k=j+l}q_iq_kq_{-j}q_{-l}$$

As for the second assertion, this is clear from the first one.
\end{proof}

The minimization problematics is very related to the CHC, and we have:

\begin{problem}
Assume that $q\in\mathbb T^N$ satisfies $\bar{q}_i=q_{-i}$.
\begin{enumerate}
\item For $N>4$, is it true that we have $\Phi>N^2$?

\item What is the $N\to\infty$ behavior of $\min\Phi-N^2$?
\end{enumerate}
\end{problem}

Let us assume now that $N=2m$ is even, and try to find the minimum of $\Phi$. 

The first observation is that we have $q_0,q_m\in\{\pm 1\}$, and that by using the symmetries in Proposition 4.1, we can always assume that we have $q_0=q_m=1$:
$$q=(1,q_1,\ldots,q_{m-1},1,\bar{q}_{m-1},\ldots,\bar{q}_1)$$

In the case particular case $N=4n$, which is of course the one we are interested in, it is convenient to write our vector in the following special way:

\begin{proposition}
In the case $N=4n$, the minimum of $\Phi$ stays the same when restricting attention to the vectors of the form
$$q=(1,a_1,\ldots,a_{n-1},b,c_{n-1},\ldots,c_1,1,\bar{c}_1,\ldots,\bar{c}_{n-1},\bar{b},\bar{a}_{n-1},\ldots,\bar{a}_1)$$
with $a,c\in\mathbb T^{n-1}$ and $b\in\mathbb T$, and with no constraints between $a,b,c$.
\end{proposition}

\begin{proof}
This is clear from the above $N=2m$ discussion, based on Proposition 4.1.
\end{proof}

The point with this writing is that the variables $b=q_n$, $\bar{b}=q_{3n}$ play a special role. Indeed, it is clear that $b^4,\bar{b}^4$ will appear with multiplicity $1$, and that $b^3,\bar{b}^3$ will not appear at all. So, our quantity $\Phi$ must have a decomposition of the following type:
$$\Phi=P_0+P_1(b+\bar{b})+P_2(b^2+\bar{b}^2)+(b^4+\bar{b}^4)$$

In other words, with $b=e^{{\rm i}\beta}$, we must have:
$$\Phi=P_0+2P_1\cos\beta+2P_2\cos 2\beta+2\cos 4\beta$$

This observation suggests a two-step approach to the minimization problem. However, since the answer to Problem 4.2 (2) is not known to us for the moment, we don't have so far any clear strategy. Of course, if the quantity $\min\Phi-N^2$ would stay $\geq 4$, then we could simply erase the $2\cos 4\beta$ term, and then easily minimize with respect to $\beta$.

By using a computer, we have found the following results:

\begin{proposition}
Let $q=(1,a_1,\ldots,a_{n-1},b,c_{n-1},\ldots,c_1,1,\bar{c}_1,\ldots,\bar{c}_{n-1},\bar{b},\bar{a}_{n-1},\ldots,\bar{a}_1)$, and denote by $\alpha_i,\beta,\gamma_i$ the arguments of the numbers $a_i,b,c_i$.
\begin{enumerate}
\item At $N=4$ we have $\min\Phi=16$, attained at $\beta=\pm\pi/2$.

\item At $N=8$ we have $\min\Phi=256/3=85.33..$, attained for instance at $\beta=\pi$, and $\alpha=\gamma=-\arccos(1/\sqrt{3})$.

\item At $N=12$ we have $\min\Phi=162$, attained for instance at $\beta=-\pi/4$, and $\alpha_1=\gamma_1=\pi/4$, $\alpha_2=\gamma_2=-2\pi/3$.
\end{enumerate}
\end{proposition}

\begin{proof}
These results were obtained using Mathematica, the idea being as follows:

(1) Here the formula of $\Phi$ is as follows, which gives the result:
$$\Phi=4(3+\cos(2 \beta))^2$$

(2) Here the formula of $\Phi$ is as follows, which once again gives the result:
\begin{eqnarray*}
\Phi
&=&170+2\cos(4\beta)+12\cos(2\alpha+2\gamma)\\
&+&8\cos(\alpha-3\gamma)+8\cos(3\alpha-\gamma)\\
&+&48\cos(\alpha+\gamma)+48\cos(\alpha-\beta-\gamma)+48\cos(\alpha+\beta-\gamma)\\
&+&24\cos(2 \beta)+24\cos(\alpha-2\beta+\gamma)+24\cos(\alpha+2\beta+\gamma)\\
&+&24\cos(2\alpha-\beta)+24\cos(2\alpha+\beta)+24\cos(\beta-2\gamma)+24\cos(\beta+2 \gamma)
\end{eqnarray*}

(3) Here the formula of $\Phi$ is quite long, and won't be given here.
\end{proof}

The above computations suggest the following statement:

\begin{conjecture}
In the case $N=4n$, the minimum of $\Phi$ over $\{q\in\mathbb T^N|\bar{q}_i=q_{-i}\}$ is the same as the minimum of $\Phi$ over vectors as above, satisfying $a=c$.
\end{conjecture}

More precisely, the fact we can assume $q_0=q_{2n}=1$ in the minimization problem is elementary, and was already explained above. The interesting phenomenon, checked at $N=8,12$, is that we can always assume that the vectors $a,c$ are equal.

We should mention that, even when assuming $a=c$, the minimization problem looks quite difficult. For instance at $N=8$, in terms of $x=a^2+\bar{a}^2$ and $y=b+\bar{b}$ we have:
$$\Phi=6x^2+4x(3y^2+6y+2)+(y^4+8y^2+48y+136)$$

This formula can be written in the following way:
$$3\Phi=2(3x+3y^2+6y+2)^2+(-15y^4-72y^3-72y^2+96y+400)$$

What is quite interesting here is that the minimum over $[-2,2]^2$ is attained at $x=-2/3,y=-2$, which are values making vanish the square in the above expression:
$$3x+3y^2+6y+2=-2+12-12+2=0$$

We do not have any explanation, or higher dimensional analogue, for this fact:

\begin{problem}
In the $a=c$ case, what are the further symmetries of the problem? Is it true that certain squares that appear naturally in the formula of $\Phi$ must vanish?
\end{problem}

Summarizing, the minimization problem for $\Phi$ is so far concerned with the construction of the good parameter space for the problem, and we have here two experimental findings so far, obtained at $N=8,12$ and $N=8$ respectively, and waiting to be understood.

\section{Critical points}

One possible way to find the minimum of $\Phi$ is by identifying its critical points, and then checking its values on these critical points to find the minimum. Also, the computation of the critical points, or at least of their symmetries, could turn to be key for some other related problems, e.g. for establishing Conjecture 4.5 above, or for finding a better parameter space for the integration problems discussed in section 6 below.

Let us first give a criterion for a vector $q\in\mathbb T^N$ to be a critical point of $\Phi$:

\begin{theorem}
Write $\Phi=\Phi_0+\ldots+\Phi_{N-1}$, with each $\Phi_i$ being given by the same formula as $\Phi$, namely $\Phi=\sum_{i+k=j+l}\frac{q_iq_k}{q_jq_k}$, but keeping the index $i$ fixed. Then:
\begin{enumerate}
\item We have $\Phi\in\mathbb R$, and $\Phi\geq N^2$, with equality in the Hadamard case.

\item The critical points of $\Phi$ are those where $\Phi_i\in\mathbb R$, for any $i$.

\item In the Hadamard case we have $\Phi_i=N$, for any $i$.
\end{enumerate}
\end{theorem}

\begin{proof}
(1) This assertion was already proved above.

(2) The first observation is that the non-constant terms in the definition of $\Phi$ involving the variable $q_i$ are the terms of the sum $K_i+\bar{K}_i$, where:
$$K_i=\sum_{2i=j+l}\frac{q_i^2}{q_jq_l}+2\sum_{k\neq i}\sum_{i+k=j+l}\frac{q_iq_k}{q_jq_l}$$

Thus if we fix $i$ and we write $q_i=e^{{\rm i}\alpha_i}$, we obtain:
$$\frac{\partial\Phi}{\partial\alpha_i} 
=4Re\left(\sum_k\sum_{i+k=j+l}{\rm i}\cdot\frac{q_iq_k}{q_jq_l}\right)
=4Im\left(\sum_{i+k=j+l}\frac{q_iq_k}{q_jq_l}\right)
=4Im(\Phi_i)$$

Now since the derivative must vanish for any $i$, this gives the result.

(3) We first perform the end of the Fourier computation in the proof of Theorem 3.5 above backwards, by keeping the index $i$ fixed. We obtain: 
\begin{eqnarray*}
\Phi_i
&=&\sum_{i+k=j+l}\frac{q_iq_k}{q_jq_l}\\
&=&\frac{1}{N}\sum_s\sum_{ijkl}w^{(i-j+k-l)s}\frac{q_iq_k}{q_jq_l}\\
&=&\frac{1}{N}\sum_sw^{si}q_i\sum_jw^{-sj}\bar{q}_j\sum_kw^{sk}q_k\sum_lw^{-sl}\bar{q}_l\\
&=&N^2\sum_sw^{si}q_i\bar{\xi}_s\xi_s\bar{\xi}_s
\end{eqnarray*}

Here we have used the formula $\xi=Fq/\sqrt{N}$. Now by assuming that we are in the Hadamard case, we have $|\xi_s|=1/\sqrt{N}$ for any $s$, and so we obtain:
$$\Phi_i=N\sum_s w^{si}q_i\bar{\xi}_s=N\sqrt{N}q_i\overline{(F^*\xi)}_i=Nq_i\bar{q}_i=N$$

This ends the proof.
\end{proof}

Let us look now at the case where $N$ is small. We denote as usual by $\mathbb Z_l$ the cyclic group formed by the $l$-th roots of unity in the complex plane. Let $w=e^{2\pi {\rm i}/3}$.

\begin{proposition}
The critical points of $\Phi$, with $q_0=1$, are as follows:
\begin{enumerate}
\item $N=2$. Here the condition is $q_1\in\mathbb Z_4$.

\item $N=3$. Here the solutions $(q_1,q_2)$ form the following $18$-element subset of $\mathbb Z_6\times\mathbb Z_6$:
$$(1,1),(w,w^2),(w^2,w),(1,w^2),(w,w),(w^2,1)$$
$$(1,w),(w,1),(w^2,w^2),(-1,-1),(-w,-w^2),(-w^2,-w)$$
$$(-1,1),(-w,w^2),(-w^2,w),(1,-1),(w,-w^2),(w^2,-w)$$
\end{enumerate}
\end{proposition}

\begin{proof}
We use the various formulae from the discussion at the end of section 3.

(1) We recall that at $N=2$ we have:
$$\Phi=6+\left(\frac{q_0}{q_1}\right)^2+\left(\frac{q_1}{q_0}\right)^2$$

The decomposition $\Phi=\Phi_0+\Phi_1$ is as follows:
\begin{eqnarray*}
\Phi_0&=&3+\left(\frac{q_0}{q_1}\right)^2\\
\Phi_1&=&3+\left(\frac{q_1}{q_0}\right)^2
\end{eqnarray*}

Thus the critical points appear at $q_1/q_0\in\{1,{\rm i},-1,-{\rm i}\}$, which gives the result.

(2) We recall from the end of section 3 that we have:
$$\Phi=15+4Re\left(\frac{q_0^3+q_1^3+q_2^3}{q_0q_1q_2}\right)$$

In terms of $a=q_0^2/(q_1q_2)$, $b=q_1^2/(q_0q_2)$, $c=q_2^2/(q_0q_1)$, we have:
$$\Phi=15+Re(a+b+c)$$

The decomposition $\Phi=\Phi_0+\Phi_1+\Phi_2$ is as follows:
\begin{eqnarray*}
\Phi_0&=&5+2a+\bar{b}+\bar{c}\\
\Phi_1&=&5+2b+\bar{a}+\bar{c}\\
\Phi_2&=&5+2c+\bar{a}+\bar{b}
\end{eqnarray*}

It follows that we have:
\begin{eqnarray*}
Im(\Phi_0)&=&2Im(a)-Im(b)-Im(c)\\
Im(\Phi_1)&=&2Im(b)-Im(a)-Im(c)\\
Im(\Phi_2)&=&2Im(c)-Im(a)-Im(b)
\end{eqnarray*}

Thus the critical point condition is simply:
$$Im(a)=Im(b)=Im(c)$$

Now recall that $a,b,c$ must satisfy $abc=1$. If these numbers are all equal, we must have $a=b=c\in\mathbb Z^3$. If not, two of them must be of the form $r,-\bar{r}$ and the third one must be $r$ or $-\bar{r}$, and from $abc=1$ we obtain that $r$ must be real, $r=\pm 1$. Thus the solutions in this latter case must be, up to permutations, of the form $(1,-1,-1)$.

Sumarizing, with $w=e^{2\pi i/3}$, the solutions are:
$$(a,b,c)\in\{(1,1,1),(w,w,w),(w^2,w^2,w^2),(1,-1,-1),(-1,1,-1),(-1,-1,1)\}$$

With our normalization $q_0=1$, the above conditions $a=q_0^2/(q_1q_2)$, $b=q_1^2/(q_0q_2)$, $c=q_2^2/(q_0q_1)$ definining $a,b,c$ become:
$$a=\frac{1}{q_1q_2},\quad b=\frac{q_1^2}{q_2},\quad c=\frac{q_2^2}{q_1}$$

We can of course neglect the third equation, the product being 1 anyway, and we get:
$$q_1^3=\frac{b}{a},\quad q_2=\frac{q_1^2}{b}$$

By plugging in the above 6 solutions for $(a,b,c)$, we obtain the result.
\end{proof}

It would be of course very interesting to work out now the case $N=4$. Here the decomposition $\Phi=\Phi_0+\Phi_1+\Phi_2+\Phi_3$ is as follows:
\begin{eqnarray*}
\Phi_0&=&7+2\left(\frac{q_0q_1}{q_2q_3}+\frac{q_0q_3}{q_1q_2}\right)+\frac{q_0^2}{q_2^2}+2\cdot\frac{q_0^2}{q_1q_3}+\left(\frac{q_0q_2}{q_1^2}+\frac{q_0q_2}{q_3^2}\right)\\
\Phi_1&=&7+2\left(\frac{q_0q_1}{q_2q_3}+\frac{q_1q_2}{q_0q_3}\right)+\frac{q_1^2}{q_3^2}+2\cdot\frac{q_1^2}{q_0q_2}+\left(\frac{q_1q_3}{q_0^2}+\frac{q_1q_3}{q_2^2}\right)\\
\Phi_2&=&7+2\left(\frac{q_2q_3}{q_0q_1}+\frac{q_1q_2}{q_0q_3}\right)+\frac{q_2^2}{q_0^2}+2\cdot\frac{q_2^2}{q_1q_3}+\left(\frac{q_0q_2}{q_1^2}+\frac{q_0q_2}{q_3^2}\right)\\
\Phi_3&=&7+2\left(\frac{q_2q_3}{q_0q_1}+\frac{q_0q_3}{q_1q_2}\right)+\frac{q_3^2}{q_1^2}+2\cdot\frac{q_3^2}{q_0q_2}+\left(\frac{q_1q_3}{q_0^2}+\frac{q_1q_3}{q_2^2}\right)
\end{eqnarray*}

The solution here seems to require substantially more work than at $N=2,3$.

The problems regarding the critical points can be summarized as follows:

\begin{problem}
What are the critical points at $N=4$? In general, do they have some interesting algebraic property? Can they be of help in proving Conjecture 4.5?
\end{problem}

Observe that we can't of course expect in general these critical points to belong to some nice cyclic group, because starting at $N=6$ we have lots of ``exotic'' cyclic $N$-roots.

In addition, we have the following conjecture, found by using a computer:

\begin{conjecture}
For a critical point the value of $\Phi_i$ depends only on the parity of $i$.
\end{conjecture}

Once again, we have here an interesting symmetry statement, to be added to the list of symmetry statements already presented above. Solving all these problems is of course very important for us, because the formula of $\Phi$ that we have so far is not very usable.

In several papers \cite{bc1}, \cite{bne}, \cite{bnz}, it has been noticed that among the different possible values of exponents, the Cauchy-Schwarz inequality for $p=1$ is the easiest to manipulate.

So, let us consider the following quantity, depending on $q\in\mathbb T^N$:
$$\Psi=\frac{1}{\sqrt{N}}||FQF^*||_1$$

With $\xi=Fq/\sqrt{N}$ and with $q_j=e^{{\rm i}\alpha_j}$, and $\theta_s=\arg(\xi_s)$, this function is:
$$\Psi=\sqrt{N}\sum_s|\xi_s|=\sum_s|\sum_jw^{sj}q_j|=\sum_{sj}w^{sj}e^{i(\alpha_j-\theta_s)}$$

The properties of $\Psi$ can be summarized as follows:

\begin{proposition}
Write $\Psi=\Psi_0+\ldots+\Psi_{N-1}$, with each $\Psi_k$ being given by the same formula as $\Psi$, namely $\Psi=\sum_{sk}w^{sk}e^{{\rm i}(\alpha_k-\theta_s)}$, but keeping the index $k$ fixed. Then:
\begin{enumerate}
\item We have $\Psi\in\mathbb R$, and $\Psi\leq N$, with equality in the Hadamard case.

\item The critical points of $\Psi$ are those where $\Psi_k\in\mathbb R$, for any $k$.

\item In the Hadamard case we have $\Psi_k=1$, for any $k$.
\end{enumerate}
\end{proposition}

\begin{proof}
This is quite similar to the proof of Theorem 5.1:

(1) For $U\in U_N$ the Cauchy-Schwarz inequality gives $||U||_1\leq N\sqrt{N}$, with equality if and only if $H=\sqrt{N}U$ is Hadamard, and this gives the result.

(2) By differentiating the formula of $\Psi$ with respect to $\alpha_k$, we obtain:
\begin{eqnarray*}
\frac{\partial\Psi}{\partial \alpha_k}
&=&Re\left(\sum_{sj}\left({\rm i}\frac{\partial \alpha_j}{\partial \alpha_k}-{\rm i} \frac{\partial \theta_s}{\partial \alpha_k}\right) w^{sj} e^{{\rm i}(\alpha_j- \theta_s)}  \right)\\
&=&-\sum_{sj}\left( \delta_{jk} -  \frac{\partial \theta_s}{\partial \alpha_k}\right)Im(w^{sj} e^{{\rm i}(\alpha_j-\theta_s)})\\
&=&-Im\left(\sum_{s}w^{sk} e^{{\rm i}(\alpha_k-\theta_s)}\right)+\sum_s\frac{\partial \theta_s}{\partial \alpha_k}Im(e^{-{\rm i}\theta_s}\xi_s)
\end{eqnarray*}

Now since the right term vanishes, this gives the result.

(3) In the Hadamard case we have $\xi_s=e^{{\rm i}\theta_s}$, and so:
$$\Psi_k=\sum_sw^{sk}e^{{\rm i}(\alpha_k-\theta_s)}=\sum_sw^{sk}q_k\bar{\xi}_s=q_k\bar{q}_k=1$$

This ends the proof.
\end{proof}

Observe that our critical point criterion is similar as well to the condition ``$S^tU$ is symmetric'' found in \cite{bc1}, with the sign matrix $S$ replaced now by the angle vector $\theta$.

As pointed out in \cite{bne}, one interesting problem, at least in the context of the orthogonal group problems considered there, is to find the joint critical points of all $p$-norms, in the whole exponent range $p\in [1,\infty)$. Indeed, as observed there, these joint critical points seem to enjoy interesting combinatorial properties. So, we have:

\begin{problem}
Is there any simple characterization of the joint critical points of all the maps of type $\Phi_p=||FQF^*||_p$?
\end{problem}

As a first remark, the case of exponents of type $p=2r$ with $r\in\mathbb N$ is quite similar to the case $p=4$ discussed above. Indeed, the key H\"older inequality in \cite{bne}, which allows to detect the eigenvalue vectors of the rescaled complex Hadamard matrices, is simply:  
$$\sum_{\Sigma i=\Sigma j}\frac{q_{i_1}\ldots q_{i_r}}{q_{j_1}\ldots q_{j_r}}\geq N^r$$

So, in this case an analogue of Theorem 5.1 is probably available. However, for arbitrary exponents $p\geq 1$ some angles like those in Proposition 5.5 above should definitely come into play, and it is not clear what the answer to the above problem should be.

\section{The moment method}

Consider a bounded function $\Theta:M\to[0,\infty)$, where $M$ is a compact manifold endowed with a probability measure. We have then the following well-known formula:
$$\max\Theta=\lim_{p\to\infty}\left(\int_M\Theta(x)^p\,dx\right)^{1/p}$$

In addition, more specialized quantities such as the exact number of maxima of $\Theta$ can be recovered via variations of this formula. So, in view of this observation, our various counting problems can be investigated by using this method:

\begin{proposition}
We have the formula
$$\min\Phi=N^3-\lim_{p\to\infty}\left(\int_{\mathbb T^N}(N^3-\Phi)^p\,dq\right)^{1/p}$$
where the torus $\mathbb T^N$ is endowed with its usual probability measure.
\end{proposition}

\begin{proof}
This follows from the above formula, with $\Theta=N^3-\Phi$. Observe that $\Theta$ is indeed positive, because $\Phi$ is by definition a sum of $N^3$ complex numbers of modulus 1. 
\end{proof}

The moment approach to the Hadamard matrix problematics goes in fact back to \cite{bc1}. In the present circulant case, the moment approach looks like a quite reasonable method. Indeed, integrating on the torus is a very simple operation: we just have to make reduced words out of the generators $q_i,\bar{q}_i$, and then the reduced words 1 contribute to the integral, and all the other words, containing $q$ variables, have integral 0.

More precisely, let us restrict attention to the problem of computing the moments of $\Phi$, which is more or less the same as computing those of $N^3-\Phi$. We have then:

\begin{proposition}
The moments of $\Phi$ are given by
$$\int_{\mathbb T^N}\Phi^p\,dq=\#\left\{ \begin{pmatrix}i_1k_1\ldots i_pk_p\\ j_1l_1\ldots j_pl_p\end{pmatrix}\Big|i_s+k_s=j_s+l_s,[i_1k_1\ldots i_pk_p]=[j_1l_1\ldots j_pl_p]\right\}$$
where the sets between brackets are by definition sets with repetition. 
\end{proposition}

\begin{proof}
This is indeed clear from the above discussison.
\end{proof}

Regarding now the real case, an analogue of Proposition 6.2 holds of course, but the combinatorics doesn't get any simpler. So, we are led into the following question:

\begin{problem}
How to compute the moments of $\Phi$, or rather of $N^3-\Phi$, in the real and in the complex case?
\end{problem}

One idea in dealing with this problem is by considering the ``enveloping sum'', obtained from $\Phi$ by dropping the condition $i+k=j+l$:

\begin{definition}
The enveloping sum of $\Phi$ is the function
$$\tilde{\Phi}=\sum_{ijkl}\frac{q_iq_k}{q_jq_l}$$
with the sum over all possible indices $i,j,k,l$.
\end{definition}

The point is that the moments of $\Phi$ appear as ``sub-quantities'' of the moments of $\tilde{\Phi}$, so perhaps the question to start with is to understand very well the moments of $\tilde{\Phi}$.

And this latter problem sounds like a quite familiar one, because we have:
$$\tilde{\Phi}=|\sum_iq_i|^4$$

We will be back to this a bit later. For the moment, let us do some combinatorics:

\begin{proposition}
We have the moment formula
$$\int_{\mathbb T^N}\tilde{\Phi}^p\,dq=\sum_{\pi\in P(2p)}\binom{2p}{\pi}\frac{N!}{(N-|\pi|)!}$$
where $\binom{2p}{\pi}=\binom{2p}{b_1,\ldots,b_{|\pi|}}$, with $b_1,\ldots,b_{|\pi|}$ being the lengths of the blocks of $\pi$.
\end{proposition}

\begin{proof}
Indeed, by using the same method as for $\Phi$, we obtain:
$$\int_{\mathbb T^N}\tilde{\Phi}(q)^p\,dq=\#\left\{ \begin{pmatrix}i_1k_1\ldots i_pk_p\\ j_1l_1\ldots j_pl_p\end{pmatrix}\Big|[i_1k_1\ldots i_pk_p]=[j_1l_1\ldots j_pl_p]\right\}$$

The sets with repetitions on the right are best counted by introducing the corresponding partitions $\pi=\ker\begin{pmatrix}i_1k_1\ldots i_pk_p\end{pmatrix}$, and this gives the formula in the statement.
\end{proof}

In order to discuss now the real case, we have to slightly generalize the above result, by computing all the half-moments of $\widetilde{\Phi}$. The result here is best formulated as:

\begin{proposition}
We have the moment formula
$$\int_{\mathbb T^N}|\sum q_i|^{2p}\,dq=\sum_kC_{pk}\frac{N!}{(N-k)!}$$
where $C_{pk}=\sum_{\pi\in P(p),|\pi|=k}\binom{p}{b_1,\ldots,b_{|\pi|}}$, with $b_1,\ldots,b_{|\pi|}$ being the lengths of the blocks of $\pi$.
\end{proposition}

\begin{proof}
This follows indeed exactly as Proposition 6.5 above, by replacing the exponent $p$ by the exponent $p/2$, and by splitting the resulting sum as in the statement.
\end{proof}

Observe that the above formula basically gives the moments of $\tilde{\Phi}$, in the real case. Indeed,  let us restrict attention to the case $N=2m$, which is the one we are interested in. Then, as explained in section 4 above, for the purposes of our minimization problem we can assume that our vector is of the following form:
$$q=(1,q_1,\ldots,q_{m-1},1,\bar{q}_{m-1},\ldots,\bar{q}_1)$$

So, we are led to the following conclusion, relating the real and complex cases:

\begin{proposition}
Consider the variable $X=q_1+\ldots+q_{m-1}$ over the torus $\mathbb T^{m-1}$.
\begin{enumerate}
\item For the complex problem at $N=m-1$, we have $\widetilde{\Phi}=|X|^4$ 

\item For the real problem at $N=2m$, we have $\widetilde{\Phi}=|2+X+\bar{X}|^4$.
\end{enumerate}
\end{proposition}

\begin{proof}
This is indeed clear from the definition of the enveloping sum $\widetilde{\Phi}$.
\end{proof}

In general, it is not clear to us what the formula of the numbers $C_{pk}$ in Proposition 6.6 is. So, as a conclusion here, let us just raise the following question:

\begin{problem}
What are the moments of the enveloping sum, $\tilde{\Phi}=|\sum_iq_i|^4$? And, once these moments are known, how to ``detect inside'' the moments of $\Phi$?
\end{problem}

Here is a random walk formulation of the problem, which might be useful:

\begin{proposition}
The moments of $\Phi$ have the following interpretation:
\begin{enumerate}
\item First, the moments of the enveloping sum $\int\widetilde{\Phi}^p$ count the loops of length $4p$ on the standard lattice $\mathbb Z^N\subset\mathbb R^N$, based at the origin.

\item $\int\Phi^p$ counts those loops which are ``piecewise balanced'', in the sense that each of the $p$ consecutive $4$-paths forming the loop satisfy $i+k=j+l$ modulo $N$.
\end{enumerate}
\end{proposition}

\begin{proof}
The first assertion follows from the formula in the proof of Proposition 6.5, and the second assertion follows from the formula in Proposition 6.2. 
\end{proof}

This statement looks quite encouraging, but passing from (1) to (2) is quite a delicate task, because in order to interpret the condition $i+k=j+l$ we have to label the coordinate axes of $\mathbb R^N$ by elements of the cyclic group $\mathbb Z_N$, and this is a quite unfamiliar operation. In addition, in the real case the combinatorics becomes more complex due to the symmetries of the parameter space, and we have no concrete results here so far.

\section{Concluding remarks}

We have seen in this paper that the Circulant Hadamard Conjecture (CHC) leads to the study of $\Phi=\sum_{i+k=j+l}\frac{q_iq_k}{q_jq_l}$ over the space $\{q\in\mathbb T^N|\bar{q}_i=q_{-i}\}$. More precisely, we have $\Phi\geq N^2$, and the CHC is equivalent to $\Phi>N^2$, for any $N>4$.

Our various results, and their further developments, are as follows:

(1) First, the study of $\Phi$ would greatly benefit from a better algebraic understanding of the circulant Hadamard matrices. Some results in this direction were obtained in sections 1-2 above. For a continuation of this work, we refer to our recent paper \cite{bh+}.

(2) A second series of algebraic problems concerns the critical points of $\Phi$, cf. sections 4-5 above. The study of these critical points falls into the ``almost Hadamard matrix'' framework from \cite{bc1}, \cite{bne}, \cite{bnz}. For some recent advances here, see \cite{bns}.

(3) Finally, the minimization problem for $\Phi$ appears to be a quite difficult task, needing substantial advances on (1,2) above. However, we have obtained some results here, via direct methods, in sections 4 and 6 above. For a continuation here, we refer to \cite{ban}.

\end{document}